\documentclass[a4paper, 11pt]{amsart}
\usepackage{amsmath, amscd}
\usepackage{amssymb, color, hyperref}

\def\doi#1{{\small\href{https://doi.org/#1}{\path{doi:#1}}}}
\def\arxiv#1{{\small\href{http://www.arxiv.org/abs/#1}{\path{arXiv:#1}}}}
\def\url#1{{\small\href{#1}{\path{#1}}}}

\theoremstyle{plain}
\newtheorem{theorem}{\bf Theorem}[section]
\newtheorem{proposition}[theorem]{\bf Proposition}
\newtheorem{lemma}[theorem]{\bf Lemma}
\newtheorem{corollary}[theorem]{\bf Corollary}
\newtheorem{conjecture}[theorem]{\bf Conjecture}
\newtheorem{open-problem}[theorem]{\bf Open Problem}

\theoremstyle{definition}

\newtheorem{definition}[theorem]{\bf Definition}
\newtheorem{remark}[theorem]{\bf Remark}

\newcommand{\N}{\mathbb N}
\newcommand{\Z}{\mathbb Z}
\newcommand{\R}{\mathbb R}
\newcommand{\Q}{\mathbb Q}

\newcommand{\C}{\mathbb C}

\newcommand{\AAP}{\text{\rm AAP }}
\newcommand{\AAMP}{\text{\rm AAMP }}
\newcommand{\BF}{\text{\rm BF}}

 \DeclareMathOperator{\ord}{ord}

 \DeclareMathOperator{\supp}{supp}

\newcommand{\red}{{\text{\rm red}}}
\renewcommand{\t}{\, | \,}

\numberwithin{equation}{section}

\begin{document}

\title{On the structure of length sets with maximal elasticity}

\author{Doniyor Yazdonov}

\address{Department of Mathematics and Scientific Computing, University of Graz, \newline 8010 Graz, Austria}
\email{doniyor.yazdonov@uni-graz.at}



\thanks{This work was supported by FWF, Project Number DOC 183-N}

\begin{abstract}
Let $H$ be a Krull monoid with finite class group $G$ and suppose that each class contains a prime divisor. Then every non-unit $a \in H$ has a factorization into atoms, say $a=u_1 \cdot\ldots \cdot u_k$ where $k$ is the factorization length and $u_1, \ldots, u_k$ are atoms of $H$. The set $\mathsf L (a)$ of all possible factorizaton lengths is the length set of $a$, and $\rho (H) = \sup \{ \max \mathsf L (a)/\min \mathsf L (a) \colon a \in H \}$ is the elasticity of $H$. We study the structure of length sets of elements with maximal elasticity and show that, in general,  these length sets are intervals.  
\end{abstract}

\maketitle

\section{Introduction}
Let $H$ be a Krull monoid or a Krull domain with class group $G$ and suppose that each class contains a prime divisor. Then each (non-zero and non-invertible) element $a \in H$ has a factorization into atoms, say $a = u_1 \cdot \ldots \cdot u_k$, where $k$ is the factorization length and $u_1, \ldots, u_k$ are atoms (irreducible elements) of $H$. The set $\mathsf L (a) \subset \N$ of all possible factorization lengths $k$ is called the \textit{length set} of $a$. Length sets and all parameters controlling their structure (such as elasticities and sets of distances) are the best investigated invariants in factorization theory (for a survey, we refer to \cite{Sc16a}). To begin with the extreme cases, we have $|\mathsf L (a)|=1$ for all relevant $a \in H$ if and only if $|G| \le 2$. If $G$ is infinite, then every finite nonempty subset of $\N_{\ge 2}$
occurs as a length set of some $a \in H$. Now suppose that $G$ is finite with $|G| \ge 3$. Then the length sets can be arbitrarily large, but they are all finite and well-structured (Proposition \ref{xy5.3}). Roughly speaking, they are sorts of generalized arithmetic progressions. In addition, it is known that almost all length sets are intervals (Proposition \ref{3.0}).

For a finite nonempty set $L \subset \N$, we denote by $\rho (L) = \max L /\min L$ the elasticity of $L$, and by $\rho (H)$, defined as the supremum of $\rho \big( \mathsf L (a) \big)$ over all $a$ of $H$,  the elasticity of $H$. It is well-known that $\rho (H) = \mathsf D (G)/2$, where $\mathsf D (G)$ denotes the Davenport constant of $G$ (see  Subsection \ref{zero-sum-sequences}). Recent research has seen much interest in the structure of length sets with maximal elasticity. It has turned out that these length sets have a much simpler structure than length sets in general (compare Propositions \ref{xy5.3} and \ref{2.3}).

We continue the investigation of length sets with maximal elasticity and proceed in the more general setting of transfer Krull monoids (see Subsection \ref{transferkrull}). In analogy to the results on general length sets, we show that almost all length sets with maximal elasticity are intervals. We formulate the first of our main results.

\newpage
\begin{theorem} \label{1.1}
Let $H$ be a transfer Krull monoid over a finite abelian group $G$ with $|G| \geq 3$. Suppose that $G$ has the following Property {\bf P}.
\begin{itemize}
\item[] {\bf P.}  There are $g_1, g_2 \in G$ and minimal zero-sum sequences $U_1, U_2$ over $G$ such that
                  \[
                  |U_1| = |U_2| = \mathsf D(G),\hspace{0.1cm} g_1g_2 \hspace{0.1cm}| \hspace{0.1cm}U_1, \hspace{0.2cm} \text{and } \hspace{0.1cm}(g_1 + g_2)\hspace{0.1cm} |\hspace{0.1cm} U_2 \,.
                  \]
\end{itemize}
Then there exists some
$a^* \in H$ such that for all elements 
$a \in H$ with $\rho(\mathsf L(a)) = \rho(H)$, the length
set  $\mathsf L(a^*a)$ is an interval with elasticity $\rho(H)$.
\end{theorem}


In general, length sets $\mathsf L (a)$ with elasticity $\rho (H)$ need not be intervals (in other words, the above statement does not hold for $a^*=1$; see Proposition \ref{2.3}). The proof of Theorem \ref{1.1} will be given in Section \ref{3}. We provide  two corollaries stating that almost all (in a combinatorial sense) length sets with maximal elasticity are intervals (Corollaries \ref{3.6} and \ref{density}).  Then, in Section \ref{4}, we study Property {\bf P}. Theorem \ref{4.2} shows that Property {\bf P} is satisfied by cyclic groups of odd order, by groups of rank two, by $p$-groups, and others. The only known groups, for which Property {\bf P} does not hold, are cyclic groups of even order. Our second main result shows that also the statement of Theorem \ref{1.1} does not hold for them. But, their length sets still satisfy a quite strong structural property.

\begin{theorem}\label{secondmaintheorem}
Let $H$ be a transfer Krull monoid over a cyclic group $G$ of even order $|G| \ge 4$. 
\begin{enumerate}
\item  There is no element $a\in H$ with maximal elasticity such that $\max\mathsf L(a)-1\in \mathsf L(a)$.

\item If $|G|+1\notin \mathbb{P}$, then there exist $a^*\in H$ and $M\in \N_0$ such that for all $a\in H$ with $\rho(\mathsf L(a))=\rho(H)$,  $\mathsf L(a^*a)\cap [\min\mathsf L(a^*a), \max\mathsf L(a^*a)-M]$ is an interval and $\rho ( \mathsf L (a^*a))$ has elasticity $\rho (H)$.
\end{enumerate}    
\end{theorem}

 The statement of Theorem \ref{secondmaintheorem}.2 does not hold if $|G|+1$ is a prime smaller than or equal to $11$, see Remark \ref{5.4}.

\section{Preliminaries}
We denote by $\N, \N_0,$ and $ \mathbb{P}$ the sets of positive integers, non-negative integers, and  prime numbers, respectively. For real numbers $a,b\in \R$, $[a,b]=\{x\in \Z:a\leq x\leq b\}$ means the discrete interval between $a$ and $b$. Let $L, L' \subset \Z$ be subsets. Then $L+L'=\{a+b:a\in L,\hspace{0.1cm} b\in L'\}$ denotes their {\it sumset}. The {\it set of distances} $\Delta(L) \subset \N$ is the set of all $d\in \N$ for which there is $a\in L$ such that $L\cap[a,a+d]=\{a,a+d\}$. 
If $L' \subset L$ is finite, then $L'$ is said to be an {\it interval of $L$} if $L' \ne \emptyset$ and $L' = L \cap [\min L', \max L']$. By an {\it interval}, we mean an interval of $\Z$, whence an interval is an arithmetic progression with difference $1$.

Let $G$ be an additively written finite abelian group.  Let $r\in \N$ and $(e_1,e_2,\ldots,e_r)$ be an $r$-tuple of non-zero elements of $G$. Then $(e_1,e_2,\ldots,e_r)$ is said to be \textit{independent} if for all $(m_1,m_2,\ldots,m_r)\in \Z^r$ an equation $m_1e_1+m_2e_2+\ldots+e_rm_r=0$ implies $m_ie_i=0$ for all $i\in [1,r]$. An independent $r$-tuple of non-zero elements $(e_1,e_2,\ldots,e_r)$ is said to be a \textit{basis} if $G=\langle e_1 \rangle \oplus\langle e_2 \rangle \oplus \ldots\oplus\langle e_r \rangle $. For $n\in \N$, we let $C_n$ denote a cyclic group of order $n$. 

\subsection{Monoids and their Arithmetic.} Throughout this paper, by a monoid, we mean a commutative cancellative semigroup with an identity. Let $H$ be a multiplicatively written monoid. We denote by $H^{\times}$ the group of invertible elements of $H$, and we denote $H_{\textup{red}}=\{aH^{\times}:a\in H\}$ the associated reduced monoid of $H$. For two elements $a,b\in H$, we say $a$ \textit{divides} $b$, written $a|b$, if $b\in aH$. The elements $a$ and $b$ are called \textit{associated}, written $a\simeq b$, if $aH=bH$. The element $a$ is called an \textit{atom} if $a=bc$ with $b,c\in H$ implies $b\in H^{\times}$ or $c\in H^{\times}$. The set of atoms of $H$ is denoted by $\mathcal{A}(H)$, and $H$ is said to be \textit{atomic} if every non-unit element of $H$ is a finite product of atoms. For a set $P$, we let $\mathcal{F}(P)$ denote the free abelian monoid with basis $P$. Every $a\in \mathcal{F}(P)$ has a unique representation in the form 
\[a=\prod_{p\in P}p^{\mathsf v_p(a)}\] where $\mathsf v_p: H\longrightarrow \N_0$ is the $p$-adic valuation of $a$. We call $|a|=\sum_{p\in P}\mathsf v_p(a)$ the \textit{length} of $a$ and $\textup{supp}(a)=\{p\in P: \mathsf v_p(a)>0\}$ the \textit{support} of $a$. 

The free abelian monoid $\mathsf Z(H)=\mathcal{F}(\mathcal{A}(H_{\textup{red}}))$ is called the \textit{factorization monoid} of $H$ and the map $\pi:\mathsf Z(H)\longrightarrow H_{\textup{red}}$, defined by $\pi(u)=u$ for all $u\in \mathcal{A}(H_{\textup{red}})$, is called the \textit{factorization homomorphism} of $H$. Let $a\in H$. Then $\mathsf Z(a)=\pi^{-1}(aH^{\times})\subset \mathsf Z(H)$ is the \textit{set of factorizations} of $a$,  $\mathsf L(a)=\{|a|:z\in \mathsf Z(a)\}\subset \N_0$ is the \textit{length set} of $a$, and 
$\mathcal{L}(H)=\{\mathsf L(a):a\in H\}$ is the \textit{system of  length sets} of $H$. By definition, the monoid $H$ is atomic if $\mathsf Z(a)\neq \emptyset$ for all $a\in H$. The monoid $H$ is said to be a \textit{\BF-monoid}, if $\mathsf L(a)$ is finite and nonempty for all $a\in H$, and we call  $H$  \textit{half-factorial} if $|\mathsf L(a)|=1$ for all $a\in H$. 

Suppose that $H$ is a \BF-monoid. To describe the structure of  $\mathcal{L}(H)$, we define the following  invariants. We denote by $\Delta(H)=\bigcup_{L\in \mathcal{L}(H)}\Delta(L)$ the \textit{set of distances} of $H$, and by $\rho(H)=\textup{sup }\{\rho(L): L\in\mathcal{L}(H)\}$ the \textit{elasticity} of $H$, and lastly, we call $\rho_k(H)=\textup{sup }\bigcup_{k\in L\in \mathcal{L}(H)}L$ the \textit{$k$-th elasticity} of $H$. We say that $H$ has $\textit{accepted elasticity}$ if $\rho(H)=\rho(\mathsf L(a))$ for some $a\in H$.

Let $H$ be an atomic monoid and suppose that $x,y\in \mathsf Z(H)$ and  $z=\textup{gcd}(x,y)$. Then we call $\mathsf{d}(x,y)=\textup{max}\{|z^{-1}x|,|z^{-1}y|\}$ the \textit{distance} between $x$ and $y$. Explicitly, if $x=u_1\cdot \ldots \cdot u_nw_1\cdot\ldots\cdot w_r$ and $y=v_1\cdot \ldots \cdot v_mw_1\cdot\ldots\cdot w_r$ with $n,m,r\in \N_0,$ \vspace{0.1cm} $u_1,\ldots, u_n, v_1,\ldots, v_m, w_1,\ldots, w_r\in \mathcal{A}(H_{\textup{red}})$ and $\{u_1,\ldots, u_n \}\cap \{v_1,\ldots, v_m\}=\emptyset$, then $\mathsf d(x,y)=\textup{max}\{n,m\}$. The distance function $\mathsf{d}: \mathsf Z(H)\times\mathsf Z(H)\longrightarrow \N_0$ is a metric. We define the \textit{catenary degree} $\mathsf c(a)$ for $a\in H$ to be the smallest $N\in \N_0\cup\{\infty\}$ such that, for any two factorizations $z,z'\in \mathsf Z(a)$, there exists a finite sequence $z=z_0,z_1,\ldots,z_k=z'\in \mathsf Z (a)$ satisfying that $\mathsf d(z_{i-1},z_i)\leq N$ for all $i\in [1,k]$. If  $z ,\,z' \in \mathsf Z (a)$ and $z \ne z'$, then 
\[
2 + \bigl| |z |-|z'| \bigr| \le \mathsf d (z, z') \,.
\]
If  $|\mathsf Z (a)| \ge 2$, then 
\begin{equation} \label{catenary-distance}
2 + \sup \Delta (\mathsf L(a)) \le \mathsf c (a) \,,
\end{equation}
whence $\mathsf c (a) \le 3$ implies that $\mathsf L (a)$ is an interval.
Furthermore, we define the \textit{catenary degree } $\mathsf c(H)$ of $H$ as  
\[
\mathsf c(H)=\textup{sup}\{\mathsf c(a): a\in H\}\in \N_0\cup \{\infty\} .
\]

\subsection{Krull Monoids} \label{zero-sum-sequences}
A monoid homomorphism $\varphi \colon H \to D$ is called a {\it divisor homomorphism} if, for all $a, b  \in H$, $\varphi (a) \t \varphi (b)$ (in $D$) implies that $a\t b$ (in $H$). A submonoid $S\subset H$ is called  \textit{saturated} if the inclusion $S \hookrightarrow H$ is a divisor homomorphism.
A monoid $H$ is a {\it Krull monoid} if one of the following equivalent conditions is satisfied (see \cite[Chapter 2]{Ge-HK06a}).
\begin{itemize}
\item[(a)] $H$ is completely integrally closed and satisfies the ascending chain condition on divisorial ideals.
\item[(b)] There is a divisor homomorphism $\varphi \colon H \to D$, where $D$ is a free abelian monoid, such that for every $ \alpha \in D$ there exists a  nonempty finite set $A\subset H$ with $\alpha = \gcd \varphi (A)$.
\item[(c)] There is a divisor homomorphism from $H$ into a free abelian monoid.
\end{itemize}
If $D$ is a domain, then $D$ is a Krull domain if and only if its monoid of nonzero elements is a Krull monoid. Condition (a) shows that every integrally closed noetherian domain is Krull. Let $H$ be a Krull monoid. Then there exists a free abelian monoid $F=\mathcal{F}(P)$ such that $H_{\textup{red}}\subset F$ is a submonoid, and the inclusion map $H_{\textup{red}}\hookrightarrow F$ is a divisor homomorphism that satisfies Condition (b). In this case, $F$ and $P$ are  uniquely determined by $H$ (up to isomorphism). Consequently, the group $$\mathcal{C}(H)=\mathsf q(F)/\mathsf q(H_{\textup{red}}),$$ where $\mathsf q(F)$ and $ \mathsf q(H_{\textup{red}})$ are the quotient groups of $F$ and $H_{\textup{red}}$, respectively, depends only on $H$ and is called the \textit{class group} of $H$. For background and further examples of Krull monoids, we refer to \cite{HK98, Ge-HK06a}.

Next, we introduce a Krull monoid having a combinatorial flavor, which plays a central role when studying the arithmetic of general Krull monoids. 

Let $G$ be a finite abelian group, and let $G_0\subset G$ be a subset. Then the elements of the free abelian monoid $\mathcal{F}(G_0)$ with basis $G_0$ are called \textit{sequences} over $G_0$. For a sequence $S=g_1\cdot \ldots \cdot g_\ell\in \mathcal{F}(G_0)$, we call $\sigma(S)=\sum_{i=1}^{\ell}g_i\in G$ the \textit{sum} of $S$, $|S| = \ell$ the length of $S$, and $\supp (S) = \{g_1, \ldots, g_{\ell} \} \subset G$ the {\it support} of $S$.  If $\varphi \colon G \to G'$ is a map from $G$ to a group $G'$, we set $\varphi (S) = \varphi (g_1) \cdot \ldots \cdot \varphi (g_l)$, and we define $-S = (-g_1) \cdot \ldots \cdot  (-g_l)$. The sequence $S$ is called 
\begin{itemize}
\item a {\it zero-sum sequence} if $\sigma (S)=0$, 

\item  \textit{zero-sumfree} if $\sum_{i \in I} g_i \ne 0$ for all $\emptyset \ne I \subset [1, \ell]$, and

\item  \textit{squarefree} if $|S|= |\supp (S)|$.
\end{itemize}
Furthermore, $\mathcal{B}(G_0)=\{S\in \mathcal{F}(G_0): \sigma(S)=0\}$ is the \textit{monoid of zero-sum sequences} over $G_0$. The monoid $\mathcal{B}(G_0)$ is viewed as a submonoid of $\mathcal{F}(G_0)$. $\mathcal{B}(G_0)$ is a finitely generated monoid, and since the inclusion $\mathcal B (G_0) \hookrightarrow \mathcal F (G_0)$ is a divisor homomorphism, Condition (c) implies that $\mathcal B (G_0)$ is a Krull monoid.   We denote $\mathcal{A}(G_0)$ the set of atoms of $\mathcal{B}(G_0)$, and we call the elements of $\mathcal{A}(G_0)$  the \textit{minimal zero-sum sequences} over $G_0$. We set \[
\mathcal{L}(G_0):=\mathcal{L}(\mathcal{B}(G_0)),\hspace{0.1cm} \Delta(G_0):=\Delta(\mathcal{B}((G_0)),\hspace{0.1cm} \rho(G_0):=\rho(\mathcal{B}(G_0))\textup{ and } \rho_k(G_0):=\rho_k(\mathcal{B}(G_0)) \,.
\]
Lastly, the {\it Davenport constant} of  $G$ is defined by 
\begin{equation*} \label{def-Davenport}
\mathsf D (G) = \max \{ |U| \colon U \in \mathcal A (G) \} \in \N.
\end{equation*}
 It is known that $\rho (\mathcal{B}(G))=\mathsf D(G)/2$ for all finite abelian groups $G$.

\medskip
\subsection{Transfer Homomorphisms and Transfer Krull Monoids.}\label{transferkrull} 
We start with the definition of transfer homomorphisms.
\begin{definition} 
 Let $H$ and $B$ be atomic monoids. Then a monoid homomorphism $\theta:H\longrightarrow B$ is called a \textit{transfer homomorphism} if the following properties are satisfied.\\
 1) $B=\theta(H)B^{\times}$ and $\theta^{-1}(B^{\times})=H^{\times}$.\\
 2) If $u\in H, \hspace{0.1cm} b,c\in B$ and $\theta(u)=bc$, then there exist $v,w\in H$ such that $u=vw, \theta(v)\simeq b$ and $\theta(w)\simeq c$.
\end{definition}

A transfer homomorphism $\theta:H\longrightarrow B$ allows us to pull back arithmetic properties from $B$ to $H$ (\cite[Proposition 3.2.3]{Ge-HK06a}). In particular, we have
\begin{equation} \label{eqtransfer}
\mathcal L (H) = \mathcal L (B) , \ \text{whence} \ \rho (H) = \rho (B) \ \text{and} \ \Delta (H) = \Delta (B) \,.
\end{equation}
A monoid $H$ (respectively, a domain $D$) is said to be  \textit{transfer Krull monoid over $G$} (respectively, a \textit{transfer Krull domain}) if  there exists a transfer homomorphism $\theta:H\longrightarrow \mathcal{B}(G)$ (respectively, $\theta:D\setminus\{0\}\longrightarrow \mathcal{B}(G)$).   It is a classic result that  a Krull monoid with class group $G$, which has a prime divisor in each class, is  a transfer Krull monoid over $G$ (\cite[Theorem 3.4.10]{Ge-HK06a}). Deep results of the last decade found  various classes of rings and monoids that are transfer Krull but fail to be Krull. They include non-commutative Dedekind domains, non-principal orders in Dedekind domains, and others (e.g., \cite{Sm13a, Ra25b, Ba-Re22a, Ba-Po25a}, and the survey \cite{Ge-Zh20a}). 

\subsection{Length Sets with Maximal Elasticity}\label{LSWME}
Our goal is to understand the structure of length sets with maximal elasticity. We begin by recalling the Structure Theorem for Length Sets (this result will be needed later on in the proof of Theorem \ref{secondmaintheorem}). To do so, we need the concepts of AAMPs (almost arithmetic multiprogressions) and of AAPs (almost arithmetic progressions). 

Let $\ell, d \in \N$, let $M\in  \N_0$, and let $\mathcal D\subset \Z$ be a subset with  $\{0,d\} \subset \mathcal D \subset [0,d]$. 
Let $L \subset \Z$ be a finite, nonempty subset. Then $L$ is an {\it almost arithmetic multiprogression} ({\rm AAMP})  with  {\it difference}  $d$,  {\it period}  $\mathcal D$, {\it length}  $\ell$ and  {\it bound}  $M$ if
		\[
		L = y + (L' \cup L^* \cup L'') \, \subset \, y + \mathcal D + d \Z , \quad \text{where  $y \in \Z$},
		\]
		\begin{itemize}
\item  $\min L^* = 0$, $L^*$ is an interval of  $\min L^* + \mathcal D + d \Z$,   $\ell$  is maximal for the property that  $\min L^* + (\ell - 1) d \in L^*$,

\item   $L' \subset [-M, -1]$, and  $L'' \subset \max L^* + [1,M]$.
\end{itemize}
We say $L$ is an {\it almost arithmetic progression}  ({\rm AAP})  with {\it difference} $d$,  {\it length} $\ell$, and  {\it bound} $M$   if it is an {\rm AAMP} with difference $d$, period  $\{0,d\}$,  length  $\ell$, and bound  $M$.

\begin{proposition}[The Structure Theorem for Length Sets] \label{xy5.3} 
Let $H$  be a transfer Krull monoid over a finite abelian group $G$. 
There exists some $M\in \N_0$ such that every $L\in \mathcal{L}(H)$ is an \AAMP with some difference $d \in \Delta (H)$ and bound $M$.
\end{proposition}

\begin{proof}
Since $\mathcal L (H) = \mathcal L (G)$ and  $\mathcal B (G)$ is finitely generated, this follows from \cite[Theorem 4.4.11]{Ge-HK06a}. \end{proof}

A realization theorem by Schmid shows that the above structural description is best possible (\cite{Sc09a}).
In order to describe the structure of length sets with maximal elasticity, we need  one more invariant.  Then we state some results as well as a standing conjecture (Conjecture \ref{3.2}) about length sets with maximal elasticity.
Let $\Delta_{\rho} (G)$ denote the set of all $d \in \N$ with the following property: 
      \begin{itemize}
      \item[] For every $k \in \N$, there is $L_k \in \mathcal L (G)$ which is an AAP with difference $d$, length $\ell \ge k$, and $\rho (L_k) = \rho (G)$. 
      \end{itemize}

\begin{conjecture} \label{3.2} 
Let $G$ be a finite abelian group with $|G|>4$. Then  $\Delta_{\rho} (G) = \{1\}$ if and only if $G$ is neither cyclic nor an elementary $2$-group.
\end{conjecture}
It is easy to see that $\Delta_{\rho} (G) \ne  \{1\}$ if $G$ is cyclic or an elementary $2$-group. The reverse implication has been confirmed, among others,  for groups of rank two, and for groups isomorphic either to $C_2\oplus C_2\oplus C_{2n}$ or to $C_{p^k}^{r}$, where $n,r\geq 2, k\geq 1$ and $p$ is a prime number with $p^k\geq 3$  (see \cite{Ba-Ge-Zh21d,Ge-Zh18a}).

\begin{proposition}[\cite{Ba-Ge-Zh21d}, Corollary 2.3] \label{2.3} Let $H$ be  a transfer Krull monoid over a finite abelian group $G$ and
suppose that $\Delta_{\rho} (G) = \{1\}$. Then there exists some $M\in \N_0$ such that every $L\in \mathcal{L}(H)$ with $\rho (L)= \mathsf D (G)/2$  is an \AAP with difference $1$ and bound $M$.
\end{proposition}

If the above statement holds for $M$, then it holds for every $M'$ with $M'  \ge M$. However, apart from some exceptional cases,  the statement does not hold for $M=0$. Comparing Proposition \ref{2.3} with Proposition \ref{xy5.3}, we observe that the structure of length sets with maximal elasticity is much more rigid than those of general length sets.
Moreover, if $G$ is   cyclic or an elementary $2$-group with $|G|>4$, then $\Delta_{\rho} (G) \ne \{1\}$ and  the statement of Proposition \ref{2.3} does not hold.

\section{Proof of Theorem \ref{1.1}} \label{3}

The goal of this section is to prove Theorem \ref{1.1}. In order to highlight its significance, we first
recall a result on general length sets. If $H = \mathcal B (G)$, for some finite abelian group $G$, then Proposition \ref{3.0} implies that the length sets of almost all elements (in a combinatorial sense)  are intervals. If $H$ is an analytic Krull monoid (in the sense of \cite{Ka17a}), then Proposition \ref{3.0} implies that the length sets of almost all  (in the sense of Dirichlet density) elements of $H$ are intervals (\cite[Theorem 9.4.11]{Ge-HK06a}). 

\begin{proposition} \label{3.0} 
Let $H$ be a  Krull monoid with  finite  class group $G$ and suppose that each class contains a prime divisor.
Then there exists some $a^* \in H$ such that, for all $a\in H$, the catenary degree  $\mathsf c (a^*a) \le 3$, whence the length set  $\mathsf L (a^*a)$ is an interval.
\end{proposition}

\begin{proof}
This easily follows from \cite[Theorem 7.6.9]{Ge-HK06a}.
\end{proof}

For the proof of Theorem \ref{1.1},   we proceed via a series of lemmas.

\begin{lemma}\label{lemmainitial}
    Let $H$ be a  \BF-monoid with accepted elasticity and let $a,b\in H$ with $\rho(\mathsf L(a))=\rho(\mathsf L(b))=\rho(H)$. Then  $\rho(\mathsf L(ab))=\rho(H)$, $\max\mathsf L(ab)=\max\mathsf L(a)+\max\mathsf L(b) $ and $
 \min\mathsf L(ab)=\min\mathsf L(a)+\min\mathsf L(b)$.
\end{lemma}

\begin{proof}
    Using the inclusion $\mathsf L(a)+\mathsf L(b)\subset \mathsf L(ab)$ for all $a,b\in H$,  we get $$\min\mathsf L(ab)\leq \min\mathsf L(a)+\min\mathsf L(b)\leq \max\mathsf L(a)+\max\mathsf L(b)\leq \max\mathsf L(ab)$$
    and hence
    \[\rho(H)\geq \rho(\mathsf L(ab))\geq \frac{\max\mathsf L(a)+\max\mathsf L(b)}{\min\mathsf L(a)+\min\mathsf L(b)}\geq \textup{min}\Bigl\{  \frac{\max\mathsf L(a)}{\min\mathsf L(a)}, \frac{\max\mathsf L(b)}{\min\mathsf L(b)}\Bigr\}=\rho(H).\]
    Thus we obtain $\rho(\mathsf L(ab))=\rho(H)$, and also $\max\mathsf L(ab)=\max\mathsf L(a)+\max\mathsf L(b) $ and $
 \min\mathsf L(ab)=\min\mathsf L(a)+\min\mathsf L(b)$.
\end{proof}

\begin{lemma}[\cite{Ge-Zh18a}, Lemma 3.2 (i)]\label{l3.2}
    Let $G$ be a finite abelian group with $|G|\geq 3$. A sequence $A\in \mathcal{B}(G)$ satisfies the equality $\rho(\mathsf L(A))=\mathsf D(G)/2$ if and only if there exist $k,l \in \N$ and $U_1,.\hspace{0.05cm}.\hspace{0.05cm}.\hspace{0.05cm},U_k, V_1,.\hspace{0.05cm}.\hspace{0.05cm}.\hspace{0.05cm},V_l\in \mathcal{A}(G)$ with $|U_1|=\ldots=|U_k|=\mathsf D(G), \hspace{0.1cm}|V_1|=\ldots=|V_l|=2$ such that $A=U_1\cdot\ldots\cdot U_k=V_1\cdot \ldots \cdot V_l $.
\end{lemma} \begin{proof}
    Suppose that $\rho(\mathsf L(A))=\mathsf D(G)/2$ for some $A\in \mathcal{B}(G)$. If $A=0^mC$ for $m\in \N_0$ and $C\in \mathcal{B}(G\setminus\{0\})$, then from the inequality
\[\frac{\mathsf D(G)}{2}=\frac{\max\mathsf L(A)}{\min\mathsf L(A)}=\frac{m+\max\mathsf L(C)}{m+\min\mathsf L(C)}\leq \frac{\max\mathsf L(C)}{\min\mathsf L(C)}\leq \frac{\mathsf D(G)}{2}\] we obtain $m=0$. Now write $A=U_1\cdot\ldots\cdot U_k=V_1\cdot\ldots\cdot V_l$, where $k=\min\mathsf L(A), l=\max\mathsf L(A)$ and $U_1,.\hspace{0.05cm}.\hspace{0.05cm}.\hspace{0.05cm},U_k,\hspace{0.05cm}V_1,\hspace{0.05cm}.\hspace{0.05cm}.\hspace{0.05cm}.\hspace{0.05cm},V_k\in \mathcal{A}(G)$. Then $l/k = \mathsf D (G)/2$ and
\[2l\leq \sum_{i=1}^{l}|V_i|=\sum_{j=1}^{k}|U_j|\leq k \mathsf D(G)=2l,\]
which implies that $|U_i|=\mathsf D(G)$ for all $i\in [1,k]$ and $|V_j|=2$ for all $j\in[1,l]$. The reverse implication is straightforward.
\end{proof}

\begin{lemma}\label{rkbjer}
Let $H$ be a \BF-monoid with accepted elasticity and let $a\in H$ be an element with maximal elasticity such that $\min\mathsf L(a)+1$ and   $\max\mathsf L(a)-1$ are lengths of $a$. Then there is an $N\in \N$ such that $\mathsf L(a^{n})$ is an interval for all $n\geq N$.
\end{lemma}
\begin{proof} Let $k:=\textup{min }\mathsf L(a)$,  $l:=\textup{max }\mathsf L(a)$, and suppose
that $[k,k+1]\cup[l-1,l]\subset\mathsf L(a)$.  Then the $(l-k)$-fold sumset  $(l-k)\mathsf L(a)$ is an interval. Since $\rho(\mathsf L(a))=\rho(H)$,   Lemma \ref{lemmainitial} implies that maximal and minimal elements of $(l-k)\mathsf L(a)$ coincide with the maximal and minimal elements of $\mathsf L(a^{l-k})$, respectively. Hence, $\mathsf L(a^{l-k})$ is an interval itself. Moreover,  $\mathsf L(a^n)$ is also an interval for all $n\geq l-k$, which again follows from Lemma \ref{lemmainitial}.\end{proof}

\begin{lemma}[\cite{Ge-HK06a}, Theorem 3.1.4 and Proposition 2.7.5]  \label{finiteness} 
 Let $H$ be a monoid such that $H_{\textup{red}}$ is finitely generated. Then the set of distances $\Delta (H)$ is finite and the elasticity  $\rho(H)$ is rational.
Moreover, if $S\subset H$ is a saturated submonoid, then $S_{\red}$ is finitely generated.
\end{lemma}

Let $\mathcal{B}_{\rho}(G)$ be the set of all zero-sum sequences with maximal elasticity and the identity element of $\mathcal{B}(G)$. Then Lemma \ref{lemmainitial} provides that  $\mathcal{B}_{\rho}(G)$ is a submonoid of $\mathcal{B}(G)$.

\begin{proposition}\label{fingen} Let $G$ be a finite abelian group. Then the submonoid  $\mathcal{B}_{\rho}(G) \subset \mathcal{B}(G)$ is finitely generated.
\end{proposition}

\begin{proof}
 Since $\mathcal B (G)$ is finitely generated, $\mathcal A (G)$ is finite,  say
 $$\{U\in \mathcal{A}(G): |U|=
\mathsf D(G)\}=\{U_1,\ldots,U_{\lambda}\}.$$ Let $$H=\{\overline{0}\}\cup\{(k_1,\ldots,k_\lambda)\in \N_0^\lambda: \rho(\prod_{i=1}^{\lambda}U_i^{k_i})=\mathsf D(G)/2)\}\subset (\N^\lambda_0, +).$$ 
Suppose that $(l_1,\ldots ,l_\lambda)\in H$ divides $(m_1,\ldots ,m_\lambda)\in H$ in $\N_0^{\lambda}$, which means that $(m_1-l_1, \ldots, m_{\lambda}-l_{\lambda}) \in \N_0^{\lambda}$. We set $s:=(\sum_{i=1}^{\lambda}m_i)\mathsf D(G)/2$   and $t:=(\sum_{j=1}^{\lambda}l_j)\mathsf D(G)/2$. Then there exist atoms $V_1,\ldots ,V_{s}$  with length two such that $$ \prod_{i=1}^{\lambda}U_i^{m_i}=\prod_{i=1}^{s}V_i \hspace{0.5cm}\textup{ and } \hspace{0.5cm} \prod_{i=1}^{\lambda}U_i^{l_i}=\prod_{i=1}^{t}V_i. $$
Hence, we have $\prod_{i=1}^{\lambda}U_i^{m_i-l_i}=\prod_{i=t+1}^{s}V_i$. By Lemma \ref{l3.2}, we infer that $\rho(\prod_{i=1}^{\lambda}U_i^{m_i-l_i})=\mathsf D(G)/2$. Therefore, $(m_1-l_1,\ldots , m_\lambda-l_\lambda)\in H$, whence  $H$ is a saturated submonoid of the additive monoid $(\N^\lambda_0, +)$. Thus, Lemma \ref{finiteness} implies that $H$ is finitely generated. Now consider the following monoid homomorphism 
\begin{equation*}\label{homom}
    \varphi: H \longrightarrow \mathcal{B}_{\rho}(G), \hspace{1cm }(k_1,\ldots, k_\lambda)\longmapsto \prod_{i=1}^{\lambda} U_i^{k_i}.
\end{equation*}
By Lemma \ref{l3.2} and by construction of $H$,  the map $\varphi$ is an epimorphism, whence $\mathcal{B}_{\rho}(G)$ is  finitely generated.
\end{proof}

Now we give the proof of Theorem \ref{1.1}.

\begin{proof}[Proof of Theorem \ref{1.1}] 
Let $H$ be a transfer Krull monoid over a finite abelian group $G$ with $|G| \geq 3$ and which satisfies Property {\bf P}.
By \eqref{eqtransfer}, it is sufficient to prove the assertion  for the monoid $H=\mathcal{B}(G)$. Let $g_1, g_2 \in G$ and let $U_1, U_2 \in \mathcal A (G)$ be such that
\[|U_1| = |U_2| = \mathsf D(G),\hspace{0.1cm} g_1g_2 \hspace{0.1cm}| \hspace{0.1cm}U_1, \hspace{0.2cm} \text{and } \hspace{0.1cm}(g_1 + g_2)\hspace{0.1cm} |\hspace{0.1cm} U_2.\]
We set $U_1=g_1g_2g_3\cdot\ldots\cdot g_l$  and $U_2=(g_1+g_2)h_2h_3\cdot\ldots\cdot h_l$, where $l=\mathsf D(G)$. Now, consider the sequence $A'=(-U_1)U_1(-U_2)U_2$. By Lemma \ref{l3.2},  this is a sequence with maximal elasticity. We aim to show that $U_1U_2$ can be expressed as a product of three atoms. Observe that
 \[U_1U_2=((g_1+g_2)g_3g_4\cdot\ldots\cdot g_l)g_1g_2h_2\cdot\ldots \cdot h_l.\]
Since the sequence $g_1g_2h_2\cdot\ldots \cdot h_l$ has length $l+1>\mathsf D(G)$, and its subsequence $h_2h_3\cdot\ldots\cdot h_l$ is a zero-sumfree sequence, it follows that this sequence can be factored into  two atoms only. Additionally, observe that the sequence $(g_1+g_2)g_3g_4\cdot\ldots\cdot g_l$ is also an atom. Hence, we conclude that $\min\mathsf L(A')+1\in \mathsf L(A')$. Moreover, using the equality
\[((-g_1)g_1)((-g_2)g_2)((-g_1-g_2)(g_1+g_2))=((g_1+g_2)(-g_1)(-g_2))((-g_1-g_2)g_1g_2)\] we deduce that $\max{\mathsf L}(A')-1\in \mathsf L(A')$. Thus, by Lemma \ref{rkbjer} there is $n\in \N$ such that $\mathsf L((A')^n)$ is an interval. Since $\mathcal{B}(G)$ is a finitely generated reduced  monoid, Lemma \ref{finiteness} implies that its set of distances $\Delta(G)$ is finite.  Let $k= \max\{\max \Delta(G), n  \}$, and define $A^*=(A')^k$. We assert that for any sequence $A\in \mathcal{B}(G)$ with maximal elasticity, $\mathsf L(A^*A)$ forms an interval. Indeed, since $\max\Delta (\mathsf L(A))\leq k$ and $\mathsf L(A^*)$ is an interval of length greater than $k$, we get that  $\mathsf L(A)+\mathsf L(A^*)$ is an interval. Lastly, applying Lemma \ref{lemmainitial}, we conclude that $\mathsf L(A^*A)$ is also an interval and has elasticity $\mathsf D(G)/2$.
\end{proof}

As applications of Theorem \ref{1.1}, we present two corollaries, which demonstrate that almost all (in a combinatorial sense) length sets with maximal elasticity are intervals. To establish the first result, we require a lemma concerning the growth behavior of the $n$-fold sumset of a finite set. Khvanskii first proved that the $n$-fold sumset of a finite set $A\subset \Z^d$ grows like a polynomial. His result was generalized and refined, and the bounds were made effective by many authors (\cite{me-im02a, gr-sh-wa23a, cu-go21a}). We cite a version due to Granville, Shakan, and Walker (\cite{gr-sh-wa23a}).

\begin{lemma}\label{Khvanskii} 
Let $A\subset \Z^d$, with $d\geq 1$, be a finite subset. Then there exist a polynomial $P_A\in \Q[x]$ and a constant $N_A$ such that $|nA|=P_A(n)$ for all $n\geq N_A$.
    
\end{lemma}

Note that there is  a canonical isomorphism from the free abelian monoid $\mathcal{F}(G)$ to $(\N_0^{|G|}, +)$ defined by \[\psi :\mathcal{F}(G)\longrightarrow \N_0^{|G|}, \hspace{0.2cm} S\longmapsto \big(\mathsf v_g(S)\big)_{g\in G}\]
We have $\mathcal{A}(\mathcal{B}_{\rho}(G))\subset \mathcal{B}_{\rho}(G)\subset \mathcal{F}(G)$. Set $\Omega:=\psi(\mathcal{A}(\mathcal{B}_{\rho}(G)))$. In particular, the monoid $\mathcal{B}_\rho(G)$ is isomorphic to the monoid $S:=[\Omega]$ generated by $\Omega$.

\begin{corollary} \label{3.6}
Let $G$ be a finite abelian group with $|G| \ge 3$. 
\begin{enumerate}
\item  For every  $k_0\in \N$, we have  
      \[
      \lim_{n\to\infty}\frac { \bigl| \bigl\{  
       A\in \mathcal{B}_{\rho}(G): \hspace{0.1cm} n-k_0 \in \mathsf L_{\mathcal{B}_{\rho}(G)}(A)   \bigr\} \bigr|} {\bigl| \bigl\{ A \in \mathcal{B}_{\rho} (G) : \hspace{0.1cm} n \in \mathsf L_{\mathcal{B}_{\rho}(G)}(A)\bigr\} \bigr|}= 1. 
      \]
      
\item If $G$ satisfies Property $\bf P$, then  
      \[
      \lim_{n\to\infty}\frac { \bigl| \bigl\{ A \in  \mathcal{B}_{\rho}(G) \, \colon \,
       \hspace{0.1cm}\mathsf L (A) \ \text{is an interval with} \ n \in \mathsf L_{\mathcal{B}_{\rho}(G)}(A)  \bigr\} \bigr|} {\bigl| \bigl\{ A \in \mathcal{B}_{\rho}(G) :  \hspace{0.1cm} n \in \mathsf L_{\mathcal{B}_{\rho}(G)}(A) \bigr\} \bigr|} = 1. 
      \]
\end{enumerate}      
\end{corollary}

\begin{proof}
1. Proposition \ref{fingen} implies that the set $\Omega\subset \N_0^{|G|}$, defined above,  is  finite. Moreover, since \[\{A\in \mathcal{B}_{\rho}(G):n\in \mathsf L _{\mathcal{B}_{\rho}(G)}(A)\}=\underbrace{\mathcal{A}(\mathcal{B}_{\rho}(G))\times \ldots
 \times \mathcal{A}(\mathcal{B}_{\rho}(G))}_{n-\textup{fold product}}\]
we get that $|\{A\in \mathcal{B}_{\rho}(G):n\in \mathsf L _{\mathcal{B}_{\rho}(G)}(A)\}|=|n\Omega|$ for all $n\in \N$. By Lemma \ref{Khvanskii}, there exist a polynomial $P\in \Q[x]$ and a constant $N$ such that $|n\Omega|=P(n)$ for all $n\geq N$. Thus

    $$\lim_{n\to\infty}\frac { \bigl| \bigl\{  
       A\in \mathcal{B}_{\rho}(G): \hspace{0.1cm} n-k_0 \in \mathsf L_{\mathcal{B}_{\rho}(G)}(A)   \bigr\} \bigr|} {\bigl| \bigl\{ A \in \mathcal{B}_{\rho} (G) : \hspace{0.1cm} n \in \mathsf L_{\mathcal{B}_{\rho}(G)}(A)\bigr\} \bigr|}=\lim_{n\to\infty}\frac{|(n-k_0)\Omega|}{|n\Omega|}= \lim_{n\to\infty}\frac{P(n-k_0)}{P(n)}=1.$$

2. Assume  that $G$ has Property $\bf P$. Then Theorem \ref{1.1} implies that there is some $A^* \in \mathcal{B}_{\rho} (G)$  such that $\mathsf L (A^*A)$ is an interval with elasticity $\mathsf D (G)/2$ for all  $A \in \mathcal{B}_{\rho} (G)$. Let $k_0\in \mathsf L_{\mathcal{B}_\rho(G)}(A^*)$. Thus
\[
\bigl\{ A^*A   : A\in \mathcal{B}_{\rho}(G) \bigr\}
\subset \bigl\{ A \in  \mathcal B_{\rho} (G) : \mathsf L (A) \ \text{is an interval}   \bigr\}, \
\]
which implies that 
$$|\{A\in \mathcal{B}_\rho(G): n-k_0\in \mathsf L_{\mathcal{B}_{\rho}(G)}(A) \}|=|\{A^*A:A\in \mathcal{B}_\rho(G), \hspace{0.1cm} n-k_0\in \mathsf L_{\mathcal{B}_{\rho}(G)}(A)\}| $$
$$\leq|\{A \in  \mathcal{B}_{\rho}(G) \, \colon \,
       \hspace{0.1cm}\mathsf L (A) \ \text{is an interval with} \ n \in \mathsf L_{\mathcal{B}_{\rho}(G)}(A)\}|, $$
whence the assertion follows from Part 1.      
\end{proof}

\begin{corollary}\label{density}
    Let $G$ be a finite abelian group with $|G|\geq 3$. 
    \begin{enumerate}
        \item For every $\alpha\in \N$, we have 
          \[
      \lim_{n\to\infty}\frac { \bigl| \bigl\{  
      B\in \mathcal{B}_\rho(G): |B|\leq n-\alpha \bigr\} \bigr|} {\bigl| \bigl\{ B\in \mathcal{B}_\rho(G): |B|\leq n \bigr\} \bigr|} = 1. 
      \]
      \item If $G$ satisfies Property $\bf P$, then  
      \[
      \lim_{n\to\infty}\frac { \bigl| \bigl\{ B\in \mathcal{B}_\rho(G): |B|\leq n, \hspace{0.1cm}\mathsf L(B) \text{ is an interval} \bigr\} \bigr|} {\bigl| \bigl\{B\in \mathcal{B}_\rho(G): |B|\leq n \bigr\} \bigr|} = 1. 
      \]
 \end{enumerate}
\end{corollary}
\begin{proof}
1. Let $|G|=m$.  For every $a=(a_1,\ldots, a_m)\in \N_0^m$, we define $|a|=a_1+\ldots +a_m$ and $X^a=X_1^{a_1}\ldots X_m^{a_m}$, where $X=(X_1,\ldots, X_m)$. We consider the algebra 
\[\C[S]:=\C[X^a: a\in \Omega]=\bigoplus_{i=0}^{\infty}\C[S]_n\subset \C[X_1,\ldots, X_m],\]
where $\C[S]_n$ is the linear space generated by monomials of degree $n$ in $\C[S]$. Hence $\C[S]$ is a finitely generated graded commutative algebra. Let the Hilbert series be 
\[H(\C[S],t):=\sum_{i=1}^{\infty}(\textup{dim }\C[S]_n)t^n.\]
The Hilbert-Serre Theorem (see \cite[Theorem 11.1]{at-mac69a} or \cite[Theorem 6.1]{Se02a}) implies that there exists a polynomial  $Q(t)\in \Z[t]$ such that  $H(\C[S],t)=\frac{Q(t)}{\prod_{i=1}^{\beta}(1-t^{m_i})}$ for some $\beta, m_1, \ldots, m_\beta\in \N$. Thus by \cite[Proposition 4.4.1]{Ri12a}, there exist $N\in \N$ and polynomials $f_0,f_1,\ldots ,f_{N-1}$ such that for all sufficiently large $n$, we have 
\[|\{a\in S: |a|=n\}|=\textup{dim }\C[S]_n=f_i(n) \hspace{0.2cm}\textup{ when }\hspace{0.2cm} n\equiv i\hspace{0.1cm}(\textup{mod N }).\]
Let $\alpha_N\in \N$  be such that $\alpha_NN\geq \alpha$. We claim that 
\[ \lim_{n\to\infty}\frac { \bigl| \bigl\{  
      a\in S: |a|=n-\alpha_NN \bigr\} \bigr|} {\bigl| \bigl\{ a\in S: |a|=n \bigr\} \bigr|}=1. \]
In fact, for every $i\in [0,N-1]$, we have 
\[ \lim_{k\to\infty}\frac { \bigl| \bigl\{  
      a\in S: |a|=kN+i-\alpha_NN \bigr\} \bigr|} {\bigl| \bigl\{ a\in S: |a|=kN+i \bigr\} \bigr|}=\lim_{k\to\infty}\frac{f_i((k-\alpha_N)N+i)}{f_i(kN+i)}=1.\]
Hence the claim follows. Next, since $$|\{B\in \mathcal{B}_\rho(G): |B|\leq n-\alpha\}|\geq |\{B\in \mathcal{B}_\rho(G): |B|\leq n-\alpha_NN\}|,$$ it suffices to show that 
\[\lim_{n\to\infty}\frac { \bigl| \bigl\{  
      a\in S: |a|\leq n-\alpha_NN \bigr\} \bigr|} {\bigl| \bigl\{ a\in S: |a|\leq n \bigr\} \bigr|}=\lim_{n\to\infty}\frac { \bigl| \bigl\{  
      B\in \mathcal{B}_\rho(G): |B|\leq n-\alpha_NN \bigr\} \bigr|} {\bigl| \bigl\{ B\in \mathcal{B}_\rho(G): |B|\leq n \bigr\} \bigr|} = 1.\]
    Let $b_n=|\{a\in S:|a|=n\}|, c_n=b_{n-\alpha_NN}, $ and $\epsilon_n=1-\frac{c_n}{b_n}$. Then our claim implies $\lim_{n\to\infty}\epsilon_n=0$. Note that $\lim_{n \to\infty }\sum_{i= 1}^{n}b_i=\infty$. Therefore 
    
        $$\lim_{n\to\infty}\frac { \bigl| \bigl\{  
      a\in S: |a|\leq n-\alpha_NN \bigr\} \bigr|} {\bigl| \bigl\{ a\in S: |a|\leq n \bigr\} \bigr|}=\lim_{n\to\infty}\frac{\sum_{i=1}^{n-\alpha_NN}b_i}{\sum_{i=1}^{n}b_i}=\lim_{n\to\infty}\frac{\sum_{i=\alpha_NN+1}^{n}c_i}{\sum_{i=1}^{n}b_i}$$
      $$=\lim_{n\to\infty}\frac{\sum_{i=\alpha_NN+1}^{n}(b_i-b_i\epsilon_i)}{\sum_{i=1}^{n}b_i}=\lim_{n\to\infty}\frac{\sum_{i=\alpha_NN+1}^{n}b_i-\sum_{i=\alpha_NN+1}^{n}b_i\epsilon_i}{\sum_{i=1}^{n}b_i}$$
$$=1-\lim_{n\to\infty}\frac{\sum_{i=\alpha_NN+1}^{n}b_i\epsilon_i}{\sum_{i=1}^{n}b_i}.$$
For every $\epsilon>0$, there exists $N_\epsilon\in \N$ with $N_\epsilon> \alpha_NN$ such that for every $i\geq N_\epsilon$, we have $\epsilon_i<\epsilon$. It follows that 
$$0\leq \lim_{n\to\infty}\frac{\sum_{i=\alpha_NN+1}^{n}b_i\epsilon_i}{\sum_{i=1}^{n}b_i}\leq \lim_{n\to\infty}\frac{\sum_{i=\alpha_NN+1}^{N_\epsilon}b_i\epsilon_i+\epsilon\sum_{i=N_\epsilon+1}^{n}b_i}{\sum_{i=1}^{n}b_i}\leq \lim_{n\to\infty}\frac{\sum_{i=\alpha_NN+1}^{N_\epsilon}b_i\epsilon_i}{\sum_{i=1}^{n}b_i}+\epsilon=\epsilon,$$
whence $\lim_{n\to\infty}\frac{\sum_{i=\alpha_NN+1}^{n}b_i\epsilon_i}{\sum_{i=1}^{n}b_i}=0$ and the assertion follows.

2. Assume  that $G$ has Property $\bf P$. Then again Theorem \ref{1.1} implies that there is some $A^* \in \mathcal{B}_{\rho} (G)$  such that $\mathsf L (A^*B)$ is an interval with elasticity $\mathsf D (G)/2$ for all  $B \in \mathcal{B}_{\rho} (G)$. Let $|A^*|=\alpha$. Thus
\[
\bigl\{ A^*B   : B\in \mathcal{B}_{\rho}(G) \bigr\}
\subset \bigl\{ B \in  \mathcal B_{\rho} (G) : \mathsf L (B) \ \text{is an interval}   \bigr\}, \
\]
which implies that $$|\{B\in \mathcal{B}_\rho(G): |B|\leq n-\alpha\}|=|\{A^*B:B\in \mathcal{B}_\rho(G), \hspace{0.1cm}|B|\leq n-\alpha\}| $$
$$\leq|\{B\in \mathcal{B}_\rho(G): |B|\leq n, \hspace{0.1cm}\mathsf L(B) \textup{ is an interval}\}|, $$
whence the assertion follows from Part 1.
\end{proof}

\section{On property \textbf{P} occurring in Theorem \ref{1.1}}\label{4}
Let  $G = C_{n_1} \oplus \ldots \oplus C_{n_r}$ with $r \in \N_0$ and $1 < n_1 \t \ldots \t n_r$, and set $\mathsf D^* (G) = 1 + \sum_{i=1}^r (n_i-1)$. 
 It only takes a few lines to check that
\[
\mathsf D^* (G) \le \mathsf D (G) \,,
\]
and it is well-known that equality holds for $p$-groups, groups of rank $r \le 2$, and others (e.g., \cite[Theorems 5.5.9 and  5.8.3]{Ge-HK06a}, \cite[page 56]{Ge-Ru09}, \cite{Sc11b}). But, there are groups with $\mathsf D^* (G) < \mathsf D (G)$ (see, for example,  \cite{Li20a} for some recent progress). However, so far
there is only one series of groups with $\mathsf D (G) > \mathsf D^* (G)$, for which the precise value of the Davenport constant is known. We summarize this result in the next lemma. 

\smallskip
\begin{lemma}[\cite{Sa-Ch14a}  Theorem 5.8 and Theorem 5.9]\label{regeir}
Let $G = C_2^4 \oplus C_{2k}$ with $k \ge 70$. 
\begin{enumerate}    
\item  We have
       \[
       \mathsf D (G) = \begin{cases}
                        \mathsf D^* (G) = 2k+4 & \ \text{if $k$ is even} \\
                       \mathsf D^*(G)+1=2k+5 & \ \text{if $k$ is odd} \,.
                       \end{cases}
       \]                
    
\item Suppose that $k$ is odd. Then    a sequence $U \in \mathcal F (G)$ is a minimal zero-sum sequence of length $\mathsf D(G)$ if and only if there exists a basis $(e_1,e_2,e_3,e_4,e)$ of $G$, with $\ord (e_2)=2$ for $i\in [1,4]$ and $\ord (e)=2k$, such that $U=e^{2k-3}V_1V_2$, where 
\begin{gather*}
        V_1 = (e_1+e_2+e_3+\frac{k-1}{2}e) \prod_{i=1}^{3}(e_i+\frac{k+1}{2}e) \quad \text{and}  \\
       V_2 = (e_1+e_2+e_3+e_4+\frac{k+1}{2}e) \prod_{i=1}^{3}(e_i+e_4+\frac{k+1}{2}e) \,. 
    \end{gather*}
\end{enumerate}    
\end{lemma}

\smallskip
\begin{theorem} \label{4.2}
Let $G$ be a finite abelian group with $|G|\geq 3$. Then Property
\begin{itemize}
\item[] {\bf P.}  There are $g_1, g_2 \in G$ and minimal zero-sum sequences $U_1, U_2$ over $G$ such that
                  \[
                  |U_1| = |U_2| =  \mathsf D(G),\hspace{0.1cm} g_1g_2 \hspace{0.1cm}| \hspace{0.1cm}U_1, \hspace{0.2cm} \text{and } \hspace{0.1cm}(g_1 + g_2)\hspace{0.1cm} |\hspace{0.1cm} U_2 \,
                  \]
\end{itemize}
is satisfied in each of the following cases:

\begin{enumerate}
\item[(a)] $G$ is a cyclic group of odd order

\item[(b)]  $G$ is not cyclic and $\mathsf D(G)=\mathsf D^*(G)$

\item[(c)] $G$ has odd order and there is some $U \in \mathcal A (G)$ of length $|U|=\mathsf D(G)$ that is not squarefree

\item [(d)] $G = C_2^4 \oplus C_{2k}$ with   $k \ge  70$.
\end{enumerate}
\end{theorem}

\smallskip
\noindent
{\it Remark.} The only known groups, which do not satisfy Property {\bf P} are cyclic groups of even order. They will be studied in Section \ref{5}. Moreover, there is no known group of odd order that does not satisfy Condition (c).

\begin{proof} 
(a) Suppose $G$ is cyclic of odd order $n \ge 3$. Let $g \in G$ be any generator. Clearly, $2g \in G$ is also a generator of the group, since the order of $G$ is odd. Now, consider $U_1 = g^n$ and $U_2 = (2g)^n$. Since both $U_1$ and $U_2$ are atoms with length $\mathsf{D}(G)$, our condition holds.

\smallskip
(b)  Suppose that $G \cong C_{n_1} \oplus \ldots \oplus C_{n_r}$ with $1 < n_1 \mid \ldots \mid n_r$ and  that $\mathsf D^*(G)=\mathsf D(G)$. Let   $(e_1,\ldots,e_r)$ be a basis of $G$ with $\ord (e_i)=n_i$ for $i \in [1,r]$ and with $r\geq 2$. Then  the sequence $U=(e_1+e_2+\ldots+e_r)\prod_{i=1}^{r}e_i^{n_i-1}$ has the property that $|U|=\mathsf D(G)$. Moreover, the map $\varphi:G\longrightarrow G$, given by $e_i\mapsto e_i$ for all $i\in [1,r-1]$ and $e_r\mapsto e_r+e_{r-1}$, is an isomorphism. Since $U$ is a minimal zero-sum sequence,  $\varphi(U)$ is as well, and therefore, $|\varphi(U)|=|U|=\mathsf D(G)$ and $e_{r-1}+e_r\in \textup{supp }(\varphi(U))$. If we choose  $U_1=U$ and $U_2=\varphi(U)$, then the condition holds.

\smallskip
(c) Suppose  $G$ has odd order and there is some $U \in \mathcal A (G)$ of length $|U|=\mathsf D(G)$ that is not squarefree. Then the map defined by the multiplication by two, $\varphi:G\longrightarrow G,\hspace{0.2cm} g\mapsto2g$ for every $g\in G$, is an isomorphism. Since $U$ is not squarefree, there exists $g\in G$ such that $g^2\hspace{0.1cm}|\hspace{0.1cm}U$. Then $\varphi(U)$ is a minimal zero-sum sequence of length $|\varphi(U)|=|U|=\mathsf D(G)$ and $2g\in \textup{supp }(\varphi(U))$. Thus, the condition holds with $U_1=U$ and $U_2=\varphi(U)$.

\smallskip
(d) Suppose $G = C_2^4 \oplus C_{2k}$ with $k \ge 70$. If $k$ is even, then $\mathsf D (G) = \mathsf D^* (G)$ by Lemma \ref{regeir}, whence Property {\bf P} holds by Condition (b). 

Now suppose that $k$ is odd. 
Let  $(e_1,e_2,e_3,e_4,e)$ be a  basis of $G$ as in Lemma \ref{regeir}. Then  $U=e^{2k-3}V_1V_2$, where $V_1$ and $V_2$ are as of Lemma \ref{regeir}, is a minimal zero-sum sequence of length $\mathsf D(G)$. One can  see that $(e_1,e_2,e_3,e_4,e_1+e_2+e_3+\frac{k+1}{2}e)$ is also a basis of $G$. Consider the map $\varphi : G\longrightarrow G$ defined by $e_i\mapsto e_i$ for $i\in [1,4]$ and $e\mapsto e_1+e_2+e_3+\frac{k+1}{2}e$. This is an isomorphism, which implies that $\varphi(U)\in \mathcal{A}(G)$ and $|\varphi(U)|=|U|=\mathsf D(G)$. Moreover, we have $e(e_1+e_2+e_3+\frac{k-1}{2}e)\hspace{0.1cm}|\hspace{0.1cm}U$ and $e_1+e_2+e_3+\frac{k+1}{2}e \hspace{0.1cm}|\hspace{0.1cm}\varphi(U)$ as desired.
\end{proof}

\medskip
We end this section by considering a Property {\bf P$^*$} of  finite abelian groups
 and its arithmetic consequences. Theorem \ref{4.3} establishes a result on the catenary degree for elements with maximal elasticity. It is analogous to the result given by Proposition \ref{3.0}, but its validity is quite restricted.  Let $H$ be an atomic monoid and let $a \in H$. Then the basic inequality \eqref{catenary-distance} shows that $\mathsf c (a) \le 3$ implies that $\mathsf L (a)$ is an interval. But the converse is far from being true. Thus, the conclusion of Theorem \ref{4.3} is stronger than that of Theorem \ref{1.1}.
 
Define Property {\bf P$^*$} as
\begin{itemize}
\item[] {\bf P$^*$.}  For every nonzero element $g \in G$, there is some $A_g \in \mathcal A (G)$ with $|A_g|= \mathsf D (G)$ \\ \phantom{{\bf P$^*$.}} \ and $g \in \supp (A_g)$.
\end{itemize}
Clearly, Property {\bf P$^*$} implies Property {\bf P}. To give an example of a group satisfying {\bf P$^*$}, suppose that  $G$ is an elementary $p$ group of rank $r$ and let $g=e_1 \in G \setminus \{0\}$. Then $e_1$ can be extended to a basis, say $(e_1, e_2, \ldots, e_r)$. Then
\[
A_g = (e_1+ \ldots + e_r) \prod_{i=1}^r e_i^{p-1}
\]
is a minimal zero-sum sequence of length $|A_g|= \mathsf D^* (G) = \mathsf D (G)$ and with $g \in \supp (A_g)$. On the other hand, finite cyclic groups satisfy Property {\bf P$^*$} if and only if they are of prime order. Similarly, groups of rank two need not satisfy Property {\bf P$^*$} (\cite[Theorem 5.8.4]{Ge-HK06a}).

\begin{theorem} \label{4.3}
Let $H$ be a transfer Krull monoid over a finite abelian group $G$, and suppose that $G$ satisfies Property {\bf P$^*$}. Then there exists some
$a^* \in H$ such that  all elements 
$a \in H$ with $\rho(\mathsf L(a)) = \rho(H)$ have catenary degree $\mathsf c (aa^*) \le 3$, whence the length
sets  $\mathsf L(a^*a)$ are intervals with elasticity $\rho(H)$.
\end{theorem}
  
\begin{proof} 
By \cite[Theorem 3.4.10]{Ge-HK06a}, it suffices to prove the assertion for the monoid $\mathcal{B}(G)$. Consider the following sequence 
\[
A^*=\prod_{g\in G \setminus \{0\}} A_g (-A_g)  \ \quad 
\] 
where the sequences $A_g$ are as in the Definition of Property {\bf P$^*$}.  Lemma \ref{l3.2} follows that the length set of $A^*$ has maximal elasticity.
Let $A \in \mathcal B (G)$ be such that $\rho ( \mathsf L (A)) = \mathsf D (G)/2$. Then $\mathsf \rho (\mathsf L(A^*A))=\mathsf D(G)/2$ by Lemma \ref{lemmainitial}.
Since $\textup{supp }(A^*)\cup \{0\}=G$,  we get $\textup{supp }(A^*A)\cup \{0\}=G$. Thus, \cite[Theorem 7.6.8]{Ge-HK06a} implies  that $\mathsf c(A^*A)\leq 3$, and consequently,  the length set $\mathsf L(A^*A)$ forms an interval. It has elasticity $\rho (H)$ by Lemma \ref{lemmainitial}.
\end{proof}

\section{On cyclic groups of even order: Proof of Theorem \ref{secondmaintheorem}}\label{5}

\smallskip

In this section, we prove Theorem \ref{secondmaintheorem}.

\begin{lemma}[\cite{Ge-Ru09}, Theorem 5.3.1]\label{lemma89}
    Let $G$ be a cyclic group of order $n\geq 2$. Then for every $k\in \N$ we have $\rho_{2k} (\mathcal B (G))=kn$ and $\rho_{2k+1}(\mathcal{B}(G))=kn+1$.

\end{lemma}

\begin{lemma}\label{Lemma342}
Let $G$ be a cyclic group of order $n \ge 3$. Then any sequence $A\in \mathcal{B}(G)$ with maximal elasticity is of the form:
\[A=\prod_{g\in G_0}(g^n(-g)^n)^{n_g}
\]
where $G_0=\{g\in G: \ord(g)=n\}$ and $n_g\in \N_0$.
\end{lemma}

\begin{proof} 
Let $A\in \mathcal{B}(G)$ and suppose that $\rho(\mathsf L(A))=\mathsf D(G)/2$. Then, by Lemma \ref{l3.2}, there exist atoms $U_1,.\hspace{0.1cm}.\hspace{0.1cm}.\hspace{0.1cm},U_m,\hspace{0.1cm}V_1,.\hspace{0.1cm}.\hspace{0.1cm}.\hspace{0.1cm},V_l\in \mathcal{A}(G)$ such that  $A=U_1\cdot\ldots \cdot U_m=V_1\cdot\ldots \cdot V_l$,   where $|U_i|=\mathsf D(G)$ and $|V_j|=2$ for all $i\in [1,m]$ and $j\in [1,l]$.  Lemma \ref{lemma89} ensures that $m$ is even,  say $m=2k$ for some $k\in \N$. By applying \cite[Lemma 1.4.9]{Ge-HK06a}, we obtain that $U_i=g_i^n$,  where $g_i\in G_0$ for all $i\in [1,k]$. Thus, we conclude that the sequence $A\in \mathcal{B}(G)$ with maximal elasticity is equal to 
$\prod_{i=1}^{k}U_i(-U_i)=\prod_{i=1}^kg_i^n(-g_i)^n$ (after reindexing if necessary).
\end{proof}

\begin{proof}[Proof of Theorem \ref{secondmaintheorem}] 
Let $H$ be a transfer Krull monoid over a cyclic group $G$ of even order $n \ge 4$. By \eqref{eqtransfer}, it is sufficient to prove the assertion  for the monoid $H=\mathcal{B}(G)$.

1. Take a sequence  $A\in \mathcal{B}(G)$ with maximal elasticity. Then by Lemma \ref{Lemma342}, the sequence $A$ is of the form: 
\begin{center}
    $A=\prod_{g\in G_0} \big(g^n(-g)^n \big)^{n_g}$
\end{center}
where $G_0=\{g\in G: \text{ord}(g)=n\}$. Now, we claim that $\max\mathsf L(A)-1\notin \mathsf L(A)$.
Let $N=\max\mathsf L(A)=\sum_{g\in G_0}nn_g$. Assume that  $N-1\in \mathsf L(A)$, which means that
    $A=\prod_{i=1}^{N-1}V_i$
for some atoms $V_1,\ldots ,V_{N-1}\in \mathcal{A}(G_0)$. Thus we get $2N=|A|=\sum_{i=1}^{N-1}|V_i|$. Note that $2\leq |V_i|\leq n$ for all $i\in[1,N-1]$. Therefore, we have two cases: 

\textit{Case 1:} Among the sets $|V_i|$, exactly one has cardinality four, while all others have cardinality two. Without loss of generality, assume $|V_i| = 2$ for all $i \in [1, N-2]$ and $|V_{N-1}| = 4$. This configuration is impossible, as having $|V_{N-1}| = 4$ implies $V_{N-1} = g (-g) h (-h)$ for some $g, h \in G_0$, which does not constitute an atom.

\textit{Case 2:}  
Exactly two of the $|V_i|$ values are equal to three, while all remaining values are equal to two. Without loss of generality, we assume $|V_i| = 2$ for all $i \in [1, N-3]$, and $|V_{N-2}| = |V_{N-1}| = 3$. Then for a generator $g\in G_0$, there exist $a, b, c \in [1,n]$  such that $ag+bg+cg=0$ and $ag,bg,cg\in G_0$, that is, $\textup{gcd}(abc,n)=1$ and $n\hspace{0.1cm}|\hspace{0.1cm}a+b+c$. But since $n$ is even while $a+b+c$ is an odd number, this is not possible.

2.  Assume that $n+1\notin \mathbb{P}$. Then there exist $a,b\geq 3$ such that $n+1=ab$. We consider the zero-sum sequence
\[
A'=g^n(ag)^n= \big(g(ag)^{n-b} \big) \big(g^{n-a}(ag) \big) \big(g^{a-1}(ag)^{b-1} \big) \,,
\]
which has the property that $\min\mathsf L((-A')A')+1\in \mathsf L((-A')A')$. Observe that the zero-sum sequence $(-A')A'$ has maximal elasticity. Thus, for any $k\in \N$ we have 
\[
[\min \mathsf L \big(((-A')A')^k \big),
\hspace{0.1cm}\min\mathsf L \big(((-A')A')^k \big)+k]\subset \mathsf L \big(((-A')A')^k \big) \,.
\]  
On the other hand, Proposition \ref{xy5.3} provides that there exists $M>0$ such that every  $L\in \mathcal{L}(G)$ is an AAMP with some difference $d$ and bound $M$.
 Let $l=M+\max \Delta(G)$, and choose $A^*=((-A')A')^l$. Then we have that $\mathsf L(A^*)$ is an AAMP with bound $M$ and $ [\min\mathsf L(A^*), \textup{ min }\mathsf L(A^*)+l]\subset \mathsf L(A^*)$. Therefore $\mathsf L(A^*)\cap
 [\min \mathsf L(A^*),\hspace{0.1cm} \max \mathsf L(A^*)-M ] $ forms an  interval. Let  $A\in \mathcal{B}(G)$ with maximal elasticity. Then  $\mathsf L(A^*A)$ has elasticity $\mathsf D (G)/2$ by Lemma \ref{lemmainitial}. Again, since $\mathsf L(A^*A)$ is an AAMP and since 
 \[
 \mathsf L(A^*A)\supset \mathsf L(A^*)+\mathsf L(A)\supset [\min \mathsf L(A^*)+\min \mathsf L(A), \hspace{0.1cm}\min \mathsf L(A^*)+\min \mathsf L(A)+l ] \,,
 \]
 we infer that $\mathsf L(A^* A)\cap
 [\min \mathsf L(A^*A),\hspace{0.1cm} \max \mathsf L(A^*A)-M ] $ forms an  interval. 
\end{proof}

\smallskip
\begin{remark} \label{5.4}
Let $G$ be a cyclic group of order $2n$ with $n \ge 2$ and suppose that $2n+1$ is a prime.

1. Let $2n \in \{4,6,10\}$. Then \cite[Theorem 3.5]{Ge-Zh18a} states that $1\notin \Delta_{\rho}(G)$. We claim that this implies that  $\min \mathsf L (A)+1 \notin \mathsf L (A)$ for every sequences $A\in \mathcal{B}(G)$ with maximal elasticity. Suppose, seeking a contradiction, that there exists a sequence $A \in \mathcal{B}(G)$ with maximal elasticity such that $\min \mathsf{L}(A) + 1 \in \mathsf{L}(A)$. By applying the same argument as in the proof of Theorem \ref{secondmaintheorem}, we deduce that there exists $M\in \N$ such that $ \mathsf{L}(A^*_k) \cap [\min \mathsf{L}(A^*_k), \max \mathsf{L}(A^*_k) - M]
$
is an interval for all sufficiently large $k$, where $A^*_k=A^k$. This is equivalent to asserting that $1 \in \Delta_{\rho}(G)$, which leads to a contradiction.

2. Let $2n \in \{12,16, 18, 22, 30, 36\}$ and let $A\in \mathcal{B}(G)$ be such that $\rho(\mathsf L(A))=\mathsf D(G)/2=n$. Then there exist atoms $U_1, \ldots ,U_k\in \mathcal{A}(G) $ with $|U_i|=\mathsf D(G)$ for all $i\in [1,k]$ such that $A=\prod_{i=1}^{k}U_i$. Observe that $\textup{min }\mathsf L(A)=k$. Assume that $k+1\in \mathsf L(A)$. Then there exist atoms $V_1,\ldots, V_{k+1}\in \mathcal{A}(G)$ such that $\prod_{i=1}^{k+1}V_i=\prod_{i=1}^{k}U_i$. By performing suitable cancellations, we obtain $\prod_{i=1}^{k'+1}V'_i=\prod_{i=1}^{k'}U'_i$, where $\{V'_1,\ldots ,V'_{k'+1}\}\subset \{V_1,\ldots ,V_{k+1}\} $ and  $\{U'_1,\ldots ,U'_{k'}\}\subset \{U_1,\ldots ,U_{k}\} $. A computer program has shown that the equation $$\prod_{i=1}^{k'+1}V'_i=\prod_{i=1}^{k'}U'_i$$ does not have a solution for $k'\leq 3$. (We refer to \url{https://zenodo.org/records/15676314} for further details).  
\end{remark}

We end with the following open question.

\begin{open-problem} \label{5.5}
Let $G$ be a cyclic group of even order and suppose that $|G|+1$ is a prime. Does there exist some $A \in \mathcal B (G)$ with $\rho \big( \mathsf L (A) \big) = \mathsf D (G)/2$ and $1+\min \mathsf L (A) \in \mathsf L (A)$?
\end{open-problem}

\medskip
\noindent
{\bf Acknowledgment.} I want to thank Alfred Geroldinger and Qinghai Zhong for their support and guidance, which were helpful in the completion of this work. I also thank Benjamin Hackl for his support in verifying the computations mentioned in Remark \ref{5.4}.

\medskip
\noindent
\textbf{Conflict of interest statement.} The author states that there is no conflict of interest.

\providecommand{\bysame}{\leavevmode\hbox to3em{\hrulefill}\thinspace}
\providecommand{\MR}{\relax\ifhmode\unskip\space\fi MR }
\providecommand{\MRhref}[2]{%
  \href{http://www.ams.org/mathscinet-getitem?mr=#1}{#2}
}
\providecommand{\href}[2]{#2}



\begin{thebibliography}{10}

\bibitem{at-mac69a} M.~Atiyah and I.~MacDonald, \emph{Introduction to Commutative Algebra}, Westview Press, 1969.

\bibitem{Ba-Ge-Zh21d}
A.~Bashir, A.~Geroldinger, and Q.~Zhong, \emph{On a zero-sum problem arising
  from factorization theory}, in {C}ombinatorial and {A}dditive {N}umber
  {T}heory {IV}, Springer, 2021, pp.~11 -- 24.

\bibitem{Ba-Po25a}
A.~Bashir and M.~Pompili, \emph{On transfer homomorphisms in commutative rings
  with zero divisors}, Commun. Algebra, to appear.

\bibitem{Ba-Re22a}
A.~Bashir and A.~Reinhart, \emph{On transfer {K}rull monoids}, Semigroup Forum
  \textbf{105} (2022), 73 -- 95.

\bibitem{cu-go21a} M.~Curran and L.~Goldmakher, \emph{Khovanskii’s Theorem and Effective Results
on Sumset Structure}, Discrete Analysis \textbf{27} (2021), 25pp.

\bibitem{Ge-HK06a}
A.~Geroldinger and F.~Halter-Koch, \emph{Non-{U}nique {F}actorizations.
  {A}lgebraic, {C}ombinatorial and {A}nalytic {T}heory}, Pure and Applied
  Mathematics, vol. 278, Chapman \& Hall/CRC, 2006.

\bibitem{Ge-Ru09}
A.~Geroldinger and I.~Ruzsa, \emph{Combinatorial {N}umber {T}heory and
  {A}dditive {G}roup {T}heory}, Advanced Courses in Mathematics - CRM
  Barcelona, Birkh{\"a}user, 2009.

\bibitem{Ge-Zh18a}
A.~Geroldinger and Q.~Zhong, \emph{Long sets of lengths with maximal
  elasticity}, Canad. J. Math. \textbf{70} (2018), 1284 -- 1318. 

\bibitem{Ge-Zh20a}
\bysame, \emph{Factorization theory in commutative monoids}, Semigroup Forum
  \textbf{100} (2020), 22 -- 51.

\bibitem{gr-sh-wa23a} A.~Granville and G.~Shakan and A.~Walker,  \emph{Effective Results on the Size and Structure of Sumsets}, Combinatorica  \textbf{43} (2023), 1139 -- 1178.


\bibitem{HK98}
F.~Halter-Koch, \emph{Ideal {S}ystems. {A}n {I}ntroduction to {M}ultiplicative
  {I}deal {T}heory}, Marcel Dekker, 1998.

\bibitem{Ka17a}
J.~Kaczorowski, \emph{Analytic monoids and factorization problems}, Semigroup
  Forum \textbf{94} (2017), 532 -- 555.

\bibitem{Se02a} S.~Lang, \emph{Algebra}, Graduate Texts in Mathematics,
  New York, Springer, 2002.


\bibitem{Li20a}
C.~Liu, \emph{On the lower bounds of {D}avenport constant}, J. Comb. Theory,
  Ser. A \textbf{171} (2020), {Paper No. 105162, 15pp}.


\bibitem{me-im02a} M.~Nathanson and I.~Ruzsa, \emph{Polynomial growth of sumsets in abelian semigroups}, Journal de Théorie des Nombres de Bordeaux \textbf{14} (2002), 553 -- 560.


 \bibitem{Ra25b}
B.~Rago, \emph{A characterization of transfer {K}rull orders in {D}edekind
  domains with torsion class group}, Canad. Bull. Math., to appear. 




\bibitem{Sa-Ch14a}
S.~Savchev and F.~Chen, \emph{Long minimal zero-sum sequences in the groups
  ${C}_2^{r-1} \oplus {C}_{2k}$}, Integers \textbf{14} (2014), Paper A23, 29pp.




\bibitem{Sc09a}
W.A. Schmid, \emph{A realization theorem for sets of lengths}, J. Number Theory
  \textbf{129} (2009), 990 -- 999.

\bibitem{Sc11b}
\bysame, \emph{The inverse problem associated to the {D}avenport constant for
  ${C}_2 \oplus {C}_2 \oplus {C}_{2n} $, and applications to the arithmetical
  characterization of class groups}, Electron. J. Comb. \textbf{18(1)} (2011),
  Paper No. 33, 42pp.

  \bibitem{Sc16a}
\bysame, \emph{Some recent results and open problems on sets of lengths of
  {K}rull monoids with finite class group}, in {M}ultiplicative {I}deal
  {T}heory and {F}actorization {T}heory, Springer, 2016, pp.~323 -- 352.



\bibitem{Sm13a}
D.~Smertnig, \emph{Sets of lengths in maximal orders in central simple
  algebras}, J. Algebra \textbf{390} (2013), 1 -- 43.


\bibitem{Ri12a} R.~Stanley, \emph{Enumerative Combinatorics}, Cambridge Studies in Advanced Mathematics, vol. 49, 2nd ed.,
 Cambridge University Press, 2012.

\end{thebibliography}

\end{document}